\documentclass[11pt]{article}
\usepackage[final]{epsfig}
\usepackage{graphics}
\usepackage[left=3cm,right=3cm, top=2.5cm,bottom=2.5cm,bindingoffset=0cm]{geometry}
\usepackage{amsmath}
\usepackage{amsfonts}
\usepackage{amsthm}
\usepackage{latexsym}
\usepackage{amssymb, mathabx}
\usepackage{pifont}

\usepackage{epstopdf, setspace}
\usepackage{multicol,multirow, enumitem}
\usepackage{hyperref, mathtools}
\newlist{longenum}{enumerate}{5}
\setlist[longenum,1]{label=\roman*)}
\setlist[longenum,2]{label=\alph*)}

\usepackage{array}
 
\usepackage{wrapfig}

\usepackage{calrsfs}
\DeclareMathAlphabet{\pazocal}{OMS}{zplm}{m}{n}

\usepackage{tikz}
\usetikzlibrary{positioning}
\usetikzlibrary{decorations.text}
\usetikzlibrary{decorations.pathmorphing}
\usetikzlibrary{decorations.markings}

\usetikzlibrary{patterns}
\usepackage{pgfplots}

\usetikzlibrary{arrows} 
\tikzset{
    >=stealth',
    pil/.style={
           ->,
           thick,
           shorten <=2pt,
           shorten >=2pt,}
}
\tikzset{->-/.style={decoration={
  markings,
  mark=at position .7 with {\arrow{>}}},postaction={decorate}}}
\tikzset{-->-/.style={decoration={
  markings,
  mark=at position .75 with {\arrow{>}}},postaction={decorate}}}
  \tikzset{->--/.style={decoration={
  markings,
  mark=at position .3 with {\arrow{>}}},postaction={decorate}}}
  \tikzset{a/.style={decoration={
  markings,
  mark=at position -1.5pt with {\arrow{>}}},shorten >=2pt, postaction={decorate}}}
\tikzset{-<-/.style={decoration={
  markings,
  mark=at position .4 with {\arrow{<}}},postaction={decorate}}}  

\tikzstyle{vertex} = [coordinate]

\newtheorem{lemma}{Lemma}[section]
\newtheorem{proposition}[lemma]{Proposition}
\newtheorem{remark-definition}[lemma]{Remark-Definition}
\newtheorem{theorem}[lemma]{Theorem}
\newtheorem{corollary}[lemma]{Corollary}

\newtheorem{proposition-conjecture}[lemma]{Proposition-conjecture}
\newtheorem{problem}[lemma]{Problem}

\theoremstyle{definition}
\newtheorem{example}[lemma]{Example}

\newtheorem{definition}[lemma]{Definition}
\newtheorem{remark}[lemma]{Remark}
\begin{document}
\newcommand{\eps}{{\varepsilon}}
\newcommand{\proofend}{\hfill$\Box$\bigskip}
\newcommand{\C}{{\mathbb C}}
\newcommand{\Q}{{\mathbb Q}}
\newcommand{\R}{{\mathbb R}}
\newcommand{\Z}{{\mathbb Z}}
\newcommand{\RP}{{\mathbb {RP}}}
\newcommand{\CP}{{\mathbb {CP}}}
\newcommand{\PP}{{\mathbb {P}}}
\newcommand{\ep}{\epsilon}
\newcommand{\G}{{\Gamma}}

\newcommand{\SDiff}{\mathrm{SDiff}}
\newcommand{\SVect}{\mathfrak{s}\mathfrak{vect}}
\newcommand{\vect}{\mathfrak{vect}}

\newcommand{\Ker}[1]{\mathrm{Ker} \, #1}
\newcommand{\grad}[1]{\mathrm{grad} \, #1}
\newcommand{\sgrad}[1]{\mathrm{sgrad} \, #1}
\newcommand{\St}[1]{\mathrm{St} \, #1}
\newcommand{\rank}[1]{\mathrm{rank} \, #1}
\newcommand{\codim}[1]{\mathrm{codim} \, #1}
\newcommand{\corank}[1]{\mathrm{corank} \, #1}
\newcommand{\sgn}[1]{\mathrm{sign} \, #1}
\newcommand{\ann}[1]{\mathrm{ann} \, #1}
\newcommand{\ind}[1]{\mathrm{ind} \, #1}
\newcommand{\Ann}[1]{\mathrm{Ann} \, #1}
\newcommand{\ls}[1]{\mathrm{span} \langle #1 \rangle}
\newcommand{\Tor}[1]{\mathrm{Tor} \, #1}
\newcommand{\diff}[1]{\mathrm{d}  #1}
\newcommand{\diffFX}[2]{ \dfrac{\partial #1}{\partial #2} }
\newcommand{\diffFXp}[2]{ \dfrac{\diff #1}{\diff #2} }
\newcommand{\diffX}[1]{ \frac{\partial }{\partial #1} }
\newcommand{\diffXp}[1]{ \frac{\diff }{\diff #1} }
\newcommand{\diffFXY}[3]{ \frac{\partial^2 #1}{\partial #2 \partial #3} }
\newcommand{\K}{\mathbb{K}}
\newcommand{\centrum}{\mathrm{Z}}
\newcommand{\Complex}{\mathbb{C}}
\newcommand{\Aut}{\mathrm{Aut}}
\newcommand{\Id}{\mathrm{E}}
\newcommand{\D}{\mathrm{D}}
\newcommand{\T}{\mathrm{T}}
\newcommand{\Cont}{\mathrm{C}}
\newcommand{\const}{\mathrm{const}}
\newcommand{\Hom}{\mathrm{H}}
\newcommand{\Ree}[1]{\mathrm{Re} \, #1}
\newcommand{\Imm}[1]{\mathrm{Im} \, #1}
\newcommand{\Tr}[1]{\mathrm{Tr} \, #1}
\newcommand{\tr}[1]{\mathrm{tr} \, #1}
\newcommand{\matrixtwobytwo}{\left(\begin{array}{|cc|}\hline 0 & 0 \\0 & 0 \\\hline \end{array}\right)}
\newcommand{\wave}{\tilde}
\newcommand{\LieBracket}{ [\, , ] }
\newcommand{\PoissonBracket}{ \{ \, , \} }
\newcommand{\g}{\mathfrak{g}}
\newcommand{\h}{\mathfrak{h}}
\newcommand{\lCal}{\mathfrak{l}}
\newcommand{\e}{\mathfrak{e}}
\newcommand{\so}{\mathfrak{so}}
\newcommand{\SO}{\mathrm{SO}}
\newcommand{\Orth}{\mathrm{O}}
\newcommand{\U}{\mathrm{U}}
\newcommand{\he}{\mathfrak{hso}}
\newcommand{\ELL}{\mathfrak{D}}
\newcommand{\hyp}{\mathfrak{D}^{h}}
\newcommand{\foc}{\mathfrak{D}^{\Complex}}
\newcommand{\sP}{\mathfrak{sp}}
\newcommand{\sL}{\mathfrak{sl}}
\newcommand{\ad}{\mathrm{ad}}
\newcommand{\Ad}{\mathrm{Ad}}
\newcommand{\zenter}{\mathrm{Z}}
\newcommand{\id}{\mathrm{id}}
\newcommand{\Ham}{\mathrm{Ham}}
\newcommand{\ham}{\mathfrak{ham}}
\newcommand{\Flux}{\mathrm{Flux}}
\newcommand{\Diffeo}{\mathrm{Diff}}
\newcommand{\halfTwist}{\mathrm{ht}}
\newcommand{\closure}{\wave}

\renewcommand{\proofname}{Proof}

\newcommand{\oneform}{\alpha}
\newcommand{\oneformtwo}{\beta}
\newcommand{\oneformthree}{\gamma}

\newcommand{\circulation}{ \mathfrak c}
\newcommand{\circulationone}{ \circulation}
\newcommand{\circulationtwo}{ \wave \circulation}
\newcommand{\orbit}{\pazocal O}

\newcommand{\Fibr}{\pazocal{F}}
\newcommand{\Fibrtwo}{\pazocal{G}}

\newcommand{\Diff}{\mathfrak{curl}}

\newcommand{\MCG}{\mathrm{Mod}}
\newcommand{\Stab}{\mathrm{Stab}}

\sloppy

\newcounter{bk}
\newcommand{\bk}[1]
{\stepcounter{bk}$^{\bf BK\thebk}$%
\footnotetext{\hspace{-3.7mm}$^{\blacksquare\!\blacksquare}$
{\bf BK\thebk:~}#1}}

\newcounter{ai}
\newcommand{\ai}[1]
{\stepcounter{ai}$^{\bf AI\theai}$%
\footnotetext{\hspace{-3.7mm}$^{\blacksquare\!\blacksquare}$
{\bf AI\theai:~}#1}}

\newcommand*\circled[1]{\tikz[baseline=(char.base)]{
            \node[shape=circle,draw,inner sep=0.8pt] (char) {#1};}}

            \mathtoolsset{showonlyrefs}


\title{Classification of Casimirs in 2D hydrodynamics}

\author{Anton Izosimov\thanks{
Department of Mathematics,
University of Toronto, Toronto, ON M5S 2E4, Canada;
e-mails: {\tt izosimov@math.toronto.edu} and \tt{khesin@math.toronto.edu}
} \,
and Boris Khesin$^*$}

\date{\it To the memory of Vladimir Igorevich Arnold}
\maketitle

\begin{abstract} 
We  describe a complete list of Casimirs for 2D Euler hydrodynamics on a surface without boundary: 
we define  generalized enstrophies which, along with  circulations,  form a complete set of invariants for 
coadjoint orbits of area-preserving diffeomorphisms on a surface. We also outline a possible extension 
of main notions to the boundary case and formulate several open questions in that setting.

\end{abstract}

\tableofcontents

\bigskip


\section{Introduction} \label{sect:intro}

The famous V.~Arnold stability criterion gives a sufficient condition for stability of a steady two-dimensional flow: 
the flow is stable provided that the second variation of the energy restricted to the set of isovorticed fields is 
sign-definite. This criterion  was generalized in many ways: to magnetohydrodynamics, to stratified fluids, 
to systems with additional symmetries, furthermore, such variations were studied in higher dimensions, perturbation methods were applied to show when instabilities arise, etc.

The knowledge of invariants of  isovorticed fields becomes of utmost importance for applying this criterion or
for developing its generalizations. The reason is that the criterion is based on defining a new functional 
as a combination of the fluid energy and those invariants. The latter are often called Casimirs of 2D flows, 
or enstrophies,  or invariants  of coadjoint orbits in two-dimensional hydrodynamics. 
More than once Arnold posed the problem of classification of isovorticed fields, as well as related problem 
of minimization of Dirichlet functional in 2D for initial functions of different topology. 

While it has been known for a long time that enstrophies are first integrals of 2D incompressible fluid flows, 
a complete classification of generic Casimirs in 2D was obtained only recently in \cite{IKM, IK}. Here we revisit 
and develop that classification by comparing it to other known classification of coadjoint orbits for 
diffeomorphism groups in one dimension.
To describe the orbit classification we first present an axillary problem of classification of simple Morse functions
with respect to area-preserving diffeomorphisms of a surface. It is convenient first to formulate 
the invariants as structures related to so-called Reeb graphs of functions.

\medskip

Recall that the motion of an inviscid incompressible  fluid filling an $n$-dimensional Riemannian manifold  
$M$  is governed by the hydrodynamical {\it Euler equation}
\begin{equation}\label{idealEuler}
\partial_t u+\nabla_u u=-\nabla p
\end{equation}
on  the divergence-free velocity field $u$ of a fluid flow in $M$. Here $\nabla_u u$ stands for  the Riemannian 
covariant derivative of the field $u$ along itself, while the function $p$  is determined by the divergence-free 
condition up to an additive constant.

In this paper we consider the case of a surface, $n=2$. In this setting, the vorticity of the fluid can be regarded 
as the 
function $F=du^\flat/\omega$, where $u^\flat$ is the 1-form metric-related to the vector field $u$ on the surface, 
and $\omega$ is the Riemannian area form. (For Euclidean metric and 
$u=u_1\partial/\partial x_1 +u_2\partial/\partial x_2$ the vorticity function is 
$F= \partial u_2/\partial x_1 -\partial u_1/\partial x_2$.) According to Kelvin's law, the vorticity function 
is ``frozen into" the incompressible flow. This fact allows one to define Casimirs, i.e. first integrals of the 
Euler equation valid for any Riemannian metric. Namely, it is well known that enstrophies, i.e. all moments
$$
m_i(F):= \int_{M}\!\! F^i \,\omega, \qquad i=0, 1,2,...
$$
of the vorticity function $F$, are Casimirs. These quantities are invariants of the natural action of the group 
of area-preserving diffeomorphisms of the surface $M$.

In the case of a  flow in a two-sphere whose Morse vorticity function has one maximum and one minimum
such enstrophy invariants form a complete set 
of Casimirs, see \cite{IK}, cf. \cite{ChSv}, while for more complicated functions and domains it is not so. 
Indeed, the set of all enstrophies is known to be incomplete for flows 
with generic vorticities: there are non-diffeomorphic vorticities with the same values of enstrophies, 
see Section \ref{sect:Casimirs}. 

In this paper we give a complete description of Casimir invariants for flows of an ideal 2D fluid 
with simple Morse vorticity functions. 
We define  generalized enstrophies in terms of measured Reeb graphs  and prove that 
they together with the set of circulations form a complete list of Casimirs in 2D hydrodynamics.

\tikzstyle{vertex} = [coordinate]

 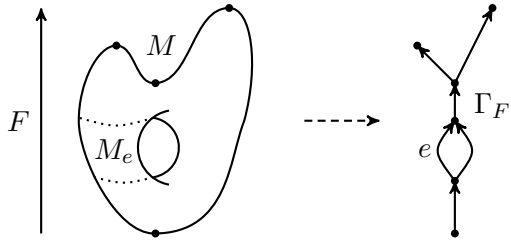
\begin{figure}[t]
\centerline{
\begin{tikzpicture}[thick]
\draw (0,1) .. controls (0,1.5) and (0.25,2) .. (0.5,2)
(0.5,2) ..controls (0.75,2) and (0.75,1.5) .. (1,1.5)
(2,2.5) ..controls (1.5,2.5) and (1.5,1.5) .. (1,1.5)
(2.2,1) .. controls (2.4,1.75) and (2.3,2.5) .. (2,2.5)
(1,-0.5) ..controls (0.5,-0.5) and (0,0.5) .. (0,1)
(1,-0.5) ..controls (2,-0.5) and (2,0.5) .. (2.2,1);
    \draw   (1.2,0.15) arc (260:100:0.5cm);
    \draw   (1,0.25) arc (-80:80:0.4cm);
    \draw [dotted] (1,0.25) .. controls (0.6, 0.15).. (0.2, 0.25);
        \draw [dotted] (1,1.04) .. controls (0.5, 0.9).. (-0.0, 1.04);
               \node at (0.47,0.6) () {$M_e$};
                       \node at (4.6,0.6) () {$e$};
\node [vertex] at (5,-0.5) (nodeA) {};
\node [vertex] at (5,0.2) (nodeB) {};
    \draw  [->] (nodeA) -- (nodeB);
\node [vertex] at (5,1) (nodeC) {};
    \draw  [->] (nodeB) .. controls +(-0.3,+0.4) .. (nodeC);
        \draw  [->] (nodeB) .. controls +(0.3,+0.4) .. (nodeC);
            \node at (5.5,1.2) (nodeZ) {$\Gamma_F$};
\node [vertex] at (5,1.5) (nodeD) {};
    \draw  [->] (nodeC) -- (nodeD);
\node [vertex] at (4.5,2) (nodeE) {};
    \draw  [->] (nodeD) -- (nodeE);
\node [vertex] at (5.5,2.5) (nodeF) {};
    \draw  [->] (nodeD) -- (nodeF);
\fill (nodeA) circle [radius=1.5pt];
\fill (nodeB) circle [radius=1.5pt];
\fill (nodeC) circle [radius=1.5pt];
\fill (nodeD) circle [radius=1.5pt];
\fill (nodeE) circle [radius=1.5pt];
\fill (nodeF) circle [radius=1.5pt];
\node [vertex] at (1.02,-0.5) (nodeAi) {};
\node [vertex] at (1.02,1.5) (nodeDi) {};
\node [vertex] at (0.5,2) (nodeEi) {};
\node [vertex] at (2,2.5) (nodeFi) {};
\fill (nodeAi) circle [radius=1.5pt];
\fill (nodeDi) circle [radius=1.5pt];
\fill (nodeEi) circle [radius=1.5pt];
\fill (nodeFi) circle [radius=1.5pt];
    \draw  [->] (-0.5,-0.5) -- (-0.5,2.5);
      \draw  [->, densely dashed] (3,1) -- (4,1);
    \node at (-0.8,1) () {$F$};
        \node at (1.1,2) () {$M$};
\end{tikzpicture}
}
\caption{Reeb graph for a height function with two maxima on a torus.}\label{torusInt}
\end{figure}

\begin{example}\label{ex:intro}
{\rm
The following example gives a glimpse of the basic constructions.
The graph $\Gamma_F$, called the Reeb graph (also called Kronrod graph),  
is the set of connected components of the levels of a height function $F$
on a  surface $M$, see Figure \ref{torusInt}. 
Critical points of $F$ correspond to the vertices  of the graph $\Gamma_F$.
This graph comes with a natural parametrization by the values of $F$. 
For a symplectic surface $M$ its area form $\omega$ induces a measure $\mu$ on the graph, 
which satisfies certain properties. 
This measured Reeb graph of the function $F$ is a complete invariant of the function $F$ with respect to the action of area-preserving diffeomorphisms of the surface:

{\bf Theorem A (= Theorem \ref{thm:sdiff-functions}).}
{\it The mapping  assigning the measured Reeb graph $\Gamma_F$ to a simple Morse function $F$  provides a one-to-one correspondence between simple Morse functions on $M$ up to  symplectomorphisms 
and measured Reeb graphs compatible with $M$.}

 \smallskip
 
 To obtain numerical invariants from this measured  graph  $\Gamma_F$ 
 one can consider  for each edge $e\in \Gamma_F$   the preimage $M_e\subset M$ bounded by 
 the corresponding critical levels of $F$. Then  infinitely many moments
$$
 I_{i,e}(F):=\int_{M_e} \!\! F^i\,\omega, ~i=0,1,2,...
$$
of the function $F$  over each $M_e$  (or, equivalently, the moments of the induced function on each edge the graph) are invariants of the $\SDiff(M)$-action, i.e., the action on  the function $F$
  by symplectomorphisms of $M$. 
  
The problem of classification of hydrodynamical Casimirs (i.e.,  invariants of coadjoint action)
for Morse coadjoint orbits includes the above problem for invariant classification of a function, 
since all invariants of vorticity are Casimirs. These two problems coincide for a sphere, as the Reeb 
graph in that case is a tree and the vorticity function fully determines the coadjoint orbit.
For surfaces of higher genus the Reeb graph has nontrivial first homology group, $\dim H_1(\Gamma_F)={\rm genus}(M)=\varkappa$ and to describe the orbit one needs also specify the circulations of the field around 
$\varkappa$ cycles on the surface.

In order to classify coadjoint orbits of the symplectomorphism group we
introduce  a notion of an anti-derivative, or circulation function,
for a Reeb graph. It turns out that such anti-derivatives
form a finite-dimensional space of dimension equal to the first Betti
number of the graph. Therefore the space of coadjoint orbits of the
symplectomorphism group of a surface is a bundle over the space of
fluid vorticities, where fiber coordinates can be thought of as
circulations, see details in Section \ref{sect:circ}.
}
\end{example}

{\bf Theorem B (= Corollary \ref{cor:Casimirs}).}
{\it 
A  complete set of Casimirs for the 2D Euler equation in a neighborhood of a Morse-type coadjoint orbit is  given 
by the moments $I_{i,e}$ for each edge $e \in \Gamma$,  $i=0,1,2,\dots$, and all circulations of the velocity $v$
over cycles in the singular levels of the vorticity function $F$ on $M$.
}
 
\medskip

\begin{remark}
{\rm 
It is interesting to compare the description of $\SDiff(M)$-orbits for a surface $M$ with the classification of coadjoint orbits of the group $\Diffeo(S^1)$ of circle diffeomorphisms~\cite{Kir1}.
Its Lie algebra is ${\vect}(S^1)$  and the (smooth) dual space ${\vect}^*(S^1)$ is identified 
with the space of quadratic differentials
on the circle, $QD(S^1):=\{F(x)(dx)^2~|~F\in C^\infty(S^1, \R)\}$.
For a generic function $F$ changing sign on the circle, a complete set of invariants is given by the ``weights"
$$
I_{a_k}(F):=\int_{a_k}^{a_{k+1}}\!\!\sqrt{|F(x)|}\,dx
$$ 
of the quadratic differential between every two consecutive  zeros ${a_k}<{a_{k+1}}$ of $F(x)$ 
on the circle $S^1$. These orbits are of finite codimension equal to the number of zeros.
In a family of functions, where two new zeros, say $a'_k$ and $a''_k$, appear between original
zeros ${a_k}$ and ${a_{k+1}}$: ${a_k}<a'_k<a''_k <a_{k+1}$, one gains two extra Casimir functions, $I_{a'_k}$ and $I_{a''_k}$, and hence the codimension of the orbit jumps up by 2.


\medskip

Similarly, for functions or coadjoint orbits of symplectomorphisms on
a 2D surface, the appearance
of a new pair of critical points, say, a saddle and a local maximum
for a function, leads
to splitting of one edge in two and, in addition to that, to
the appearance of a new edge in the corresponding Reeb graph, and
hence to two new families of Casimirs related to those extra edges, as
in Example \ref{ex:intro}.
}
\end{remark}

Note that for the action of the subgroup consisting of symplectomorphisms in the connected component 
of the identity,   one encounters additional discrete invariants related to pants decompositions 
  and possible projections of the surface to the graph.
  For the group of Hamiltonian diffeomorphisms the above set of orbit invariants is supplemented 
  by fluxes of diffeomorphisms across certain cycles on the surface $M$, see details in \cite{IKM}.

\medskip

Motivation for this type of  classification problems is coming from fluid dynamics. 
For instance, steady fluid flows are conditional extrema of the energy functional on the sets of isovorticed fields, 
so Casimirs allow one to single out such sets in order to introduce appropriate Lagrange multipliers. 
Furthermore, Casimirs in fluid dynamics are a cornerstone of the energy-Casimir method 
for the study of hydrodynamical stability, see e.g. \cite{Arn66}. 

In addition to hydrodynamics,
the results can be also used for the extension of the orbit method to infinite-dimensional groups 
of 2D diffeomorphisms. According to this method, adjacency of  coadjoint orbits of a group
or its central extension mimics families of appropriate representations of the corresponding group. 
This methods turned out to be effective for affine groups and the Virasoro-Bott group, so one may hope 
to apply it to 2D diffeomorphisms and current groups as well. Finally, note that all objects in the present 
paper are infinitely smooth (see the case of  finite smoothness in \cite{IKM}).
To the best of our knowledge, a complete description of Casimirs in 2D fluid dynamics has not previously appeared 
in the literature in a self-contained form, while various partial results could be found 
in \cite{AK, Serre, Shn, Yoshida, IK}. In the last section we present a few examples, show how the main
notions can be extended to the case of surfaces with boundary, emphasize the main difficulties and 
formulate open questions in the latter setting.

\bigskip

{\bf Acknowledgements:} A part of this work was completed while B.K. held a  Weston Visiting Professorship at the 
Weizmann Institute of Science. B.K. is grateful to the Weizmann Institute for its kind hospitality and support. 
A.I. and B.K. were partially supported by an NSERC research grant.

\medskip


\section[The hydrodynamical Euler equation]{The hydrodynamical Euler equation} \label{sect:euler}

\subsection{Geodesic and Hamiltonian frameworks of the Euler equation}

Consider an inviscid incompressible  fluid filling a closed (i.e., compact and without boundary) $n$-dimensional Riemannian manifold  $M$ 
with the Riemannian volume form $\mu$.
Arnold  \cite{Arn66} showed that the Euler equation can be regarded as an 
equation of the geodesic flow on the group $\SDiff(M):=\{\phi\in \Diffeo(M)~|~\phi^*\mu=\mu\}$ 
of volume-preserving diffeomorphisms of $M$ with respect to a right-invariant metric on the group 
given at the identity by the $L^2$-norm of the fluid's velocity field.
This geodesic description implies  the following Hamiltonian framework for the Euler equation.  
Consider the (smooth) dual space $\mathfrak g^*={\SVect}^*(M)$ 
to the space $\mathfrak g={\SVect}(M)=\{u\in {\vect}(M) \mid L_u\mu=0\}$ of divergence-free vector fields on $M$.
This dual space  has a natural  description 
as the  space of cosets  $\mathfrak g^*=\Omega^1(M) / \diff \Omega^0(M)$, where $\Omega^k(M)$ is the space of smooth $k$-forms on $M$.
For a 1-form $\alpha$ on $M$ its coset of 1-forms is 
$$
[\alpha]=\{\alpha+df\,|\,\text{ for all } f\in C^\infty(M)\}\in \Omega^1(M) / \diff \Omega^0(M)\,.
$$
The pairing between cosets and divergence-free vector fields is given by 
$\langle [\alpha],u\rangle:=\int_M\alpha(u)\,\omega$ for any  field $u\in {\SVect}(M)$. (This pairing is well-defined on cosets because the latter integral vanishes for any exact $1$-form $\alpha$ and any $u \in \SVect(M)$.)
The coadjoint action of the group $\SDiff(M)$ on the dual 
 $\mathfrak g^*$ is given by the change of coordinates in (cosets of) 1-forms on $M$ 
 by means of volume-preserving diffeomorphisms.
 
The Riemannian metric $(\,,)$ on the manifold $M$ allows one to
identify the Lie algebra and its (smooth) dual by means of the so-called inertia operator:
given a vector field $u$ on $M$  one defines the 1-form $\alpha=u^\flat$ 
as the pointwise inner product with  the velocity field $u$:
$u^\flat(v): = (u,v)$ for all $v\in T_xM$. Note also  that divergence-free fields $u$ correspond to co-closed 1-forms $u^\flat$.
The Euler equation \eqref{idealEuler} rewritten on 1-forms $\alpha=u^\flat$ is
$\partial_t \alpha+L_u \alpha=-dP$
for  an appropriate function $P$ on $M$.
In terms of the cosets of 1-forms $[\alpha]$, the Euler equation on the dual space $\mathfrak g^*$ takes the form
\begin{equation}\label{1-forms}
\partial_t [\alpha]+L_u [\alpha]=0\,.
\end{equation}
The  Euler equation \eqref{1-forms} on $\mathfrak g^*=\SVect^*(M)$ turns out to be a Hamiltonian equation with the
Hamiltonian functional  $\mathcal H$ given by the fluid's kinetic energy,
${\mathcal H}([\alpha])= \frac 12\int_M(u,u)\,\mu$ 
for $\alpha=u^\flat$.
The corresponding Poisson structure is given by  the 
natural linear Lie-Poisson bracket on the dual space $\mathfrak g^*$ 
of the Lie algebra $\mathfrak g$, see details in \cite{Arn66, AK}.
The corresponding Hamiltonian operator is given by the Lie algebra coadjoint action ${\rm ad}^*_u$, 
which  in the case of the diffeomorphism group corresponds to the Lie derivative: ${\rm ad}^*_u=L_u$.
Its symplectic leaves are coadjoint orbits of the corresponding group $\SDiff(M)$.
All invariants of the coadjoint action, also called Casimirs,  are first integrals of the Euler equation
for {\it any choice} of Riemannian metric.
The main result of this paper is a complete characterization of 
Casimirs for the 2D Euler equation on closed surfaces, see Section \ref{sect:Casimirs}.

\medskip

\subsection{Vorticity and Casimirs of the 2D Euler equation}

Recall that  according to the Euler equation  \eqref{1-forms}  the 
coset of 1-forms $[\alpha]$ evolves by a volume-preserving change of coordinates, 
i.e. during the Euler evolution it remains in the same coadjoint orbit in $\mathfrak g^*$.
Introduce the {\it vorticity 2-form} $\xi:=\diff u^\flat$ as the differential of the 1-form 
$\alpha=u^\flat$ and note  that  the vorticity exact 2-form is well-defined for cosets $[\alpha]$: 
1-forms $\alpha$ in the same coset have equal vorticities $\xi=\diff \alpha$. 
The corresponding Euler equation assumes the vorticity (or Helmholtz) form
\begin{equation}\label{idealvorticity}
\partial_t \xi+L_u \xi=0\,,
\end{equation}
which means that the vorticity form is transported by (or ``frozen into") the fluid flow (Kelvin's theorem).
The definition of vorticity $\xi$ as an exact 2-form $\xi=\diff u^\flat$ makes sense for a manifold $M$ 
of any dimension. 
In the case of two-dimensional oriented surfaces $M$   the group $\SDiff(M)$ of 
volume-preserving diffeomorphisms of $M$ coincides with the group ${\rm Symp}(M)$ 
of symplectomorphisms of $M$  with the area form $\mu=\omega$ given by the symplectic structure. 
In what follows, in 2D our main object of consideration is
is the vorticity function $F$ related to the vorticity 2-form $\xi=F\omega$ by means of the symplectic structure. 

\medskip


\begin{remark}
The fact that the vorticity 2-form $\xi$ is ``frozen into" the incompressible flow allows one to define first integrals of the  hydrodynamical Euler equation valid for any Riemannian metric on $M$. 
In 2D the Euler equation on $M$ is known to possess infinitely many so-called {\it enstrophy invariants}
$m_\lambda(F):=\int_M \!\lambda(F)\,\omega\,,$ where 
$\lambda(F)$ is an arbitrary function of the vorticity function $F$. In particular,  the enstrophy moments
$m_i(F):=\int_M F^i\,\omega$ 
are conserved quantities for any $i\in \mathbb Z_{\ge 0}$.
These Casimir invariants are  fundamental in the study of nonlinear stability of 2D flows, 
and in particular, were  the basis for Arnold's stability criterion in ideal hydrodynamics, see \cite{Arn66, AK}. 
In the energy-Casimir method one studies the second variation of the energy functional 
with an appropriately chosen combination of Casimirs. 

In the case of a flow in an annulus with a vorticity function without critical points such invariants, along with the 
circulation along one of the two boundary components,  form a complete set 
of Casimirs \cite{ChSv}, while for more complicated functions and domains it is not so, 
see Section \ref{sect:Casimirs}. 
In Section \ref{sect:Casimirs} we give a complete description of Casimirs in the general setting of Morse 
vorticity functions on two-dimensional surfaces.
\end{remark}


\section{Coadjoint orbits of the symplectomorphism group}\label{sdiffOrbitsSection}

Before classifying coadjoint orbits of the symplectomorphism group we solve the problem of finding a complete 
invariant for a function on a closed symplectic surface. (See Section \ref{sec:bcmf} for the case with boundary.)

\subsection{Simple Morse functions and measured Reeb graphs}
\begin{definition}\label{def:smf}
Let $M$ be a closed connected surface.
A Morse function  $F \colon M \to \R $ is called \textit{simple} if any $F$-level contains at most one critical point.
(Here and below under $F$-level we mean a connected component of the set $F = \const$.)
\end{definition}

With each simple Morse function $F \colon M \to \R$, one can associate a graph. This graph $\Gamma_F$ is defined as the space of $F$-levels with 
the induced quotient topology. Each vertex of this graph
corresponds to a critical level of the function $F$. 
The function $F$ on $M$ descends to a function $f$ on the graph $ \Gamma_F$. 
It is also convenient to assume that $\Gamma_F$ is oriented: edges are oriented in the direction of increasing $f$. 

\begin{example}
{\rm
Figure \ref{torusInt} shows level curves of a simple Morse function on a torus and the corresponding 
graph $\Gamma_F$.
}
\end{example}

\begin{definition}\label{RGDef}
A \textit{Reeb graph} $(\Gamma,  f)$ is an oriented connected finite graph  $\Gamma$ with a continuous function 
$f \colon \Gamma \to \R$  which satisfy the following properties.
\begin{longenum}
\item All vertices of $\Gamma$ are either $1$-valent or $3$-valent. 
\item For each $3$-valent vertex, there are either two incoming and one outgoing edge, or vice versa. 
\item The function $f$ is strictly monotonous on each edge of $\Gamma$, and the
edges of $\Gamma$ are oriented towards the direction of increasing $f$.\end{longenum}
\end{definition}
It is a standard result from Morse theory that the graph $\Gamma_F$ associated with a simple Morse function 
$F \colon M \to \R$ on an orientable connected surface $M$ is a Reeb graph in the sense of Definition \ref{RGDef}. 
We will call this graph the {Reeb graph} of the function $F$. Note that Reeb graphs classify simple Morse functions on $M$ up to diffeomorphisms.
\par
In what follows, we assume that the surface $M$ is endowed with an area (i.e., symplectic) form $\omega$. 
We are interested in the classification problem for simple Morse functions up to area-preserving (i.e., symplectic) 
diffeomorphisms. It turns out that this classification can be given in terms of so-called \textit{log-smooth measures} 
on Reeb graphs.

 \begin{definition}
 Let $\Gamma$ be a Reeb graph. Assume that $e_0, e_1$, and $e_2$ are three edges of  $\Gamma$ which 
 meet at a $3$-valent vertex $v$. Then $e_0$ is called the \textit{trunk} of $v$, and $ e_1, e_2$ are called 
\textit{branches} of $v$ if either $e_0$ is an outgoing edge for $v$, and $e_1, e_2$ are its incoming edges, 
or vice versa.
\end{definition}

\begin{definition}\label{measureproperty} A measure $\mu$ on a Reeb graph $(\Gamma,f)$ is called \textit{log-smooth} if it has the following properties:
\begin{longenum}
\item It has a $C^\infty$-smooth non-vanishing density $\diff \mu / \diff f$ at interior points and $1$-valent vertices of $\Gamma$.
\item At $3$-valent vertices, the measure $\mu$ has logarithmic singularities. More precisely, consider
a $3$-valent vertex  $v$ of $\Gamma$. Without loss of generality assume that $f(v)=0$ (if not, we replace 
$f$ by $\tilde f(x):=f(x)-f(v)$).
Let $e_0$ be the trunk of $v$, and let $e_1, e_2$ be the branches of $v$.
 Then there exist functions $ \psi, \eta_0, \eta_1,\eta_2$ of one variable, $C^\infty$-smooth in the neighborhood of the origin $0\in \R$ and such that for any point $x \in e_i$ sufficiently close to $v$, we have
\begin{equation*}
\mu([v,x])= \,\eps_i\psi(  f(x) )\ln |   f (x)| + \eta_i(   f(x) ),
\end{equation*}
where $\eps_0 = 2$, $\eps_1 = \eps_2 =  -1$,  $\psi(0) = 0$, $\psi'(0) \neq 0$,
and $\eta_0 + \eta_1 + \eta_2 = 0$.
\end{longenum}
\end{definition}
\begin{definition}
 A Reeb graph $(\Gamma,  f)$ endowed with a log-smooth measure $\mu$ is called a \textit{measured Reeb graph}.
 \end{definition}
  
  If a surface $M$ is endowed with an area form $\omega$, then the Reeb graph $\Gamma_F$ of any simple Morse function $F \colon M \to \R$ has a natural structure of a measured Reeb graph. The measure $\mu$ on $\Gamma_F$ is defined as the pushforward of the area form on $M$ under the natural projection $\pi \colon M \to \Gamma_F$.\par
  It turns out that there is a one-to-one correspondence between simple Morse functions on $M$, considered up to symplectomorphisms, and measured Reeb graphs satisfying the following natural compatibility conditions:

\begin{definition}\label{def:compatible}
{\rm
Let $M$ be a connected closed surface endowed with a symplectic form $\omega$.
A measured Reeb graph  $(\Gamma,  f, \mu)$ is \textit{compatible with} $(M, \omega)$ if
$(i)$  the dimension of $ \Hom_1(\Gamma, \R)$  is equal to the genus 
of  $M$ and $(ii)$  the volume of $\Gamma$ with respect to the measure  $\mu$ 
is equal to the volume of $M$: $\int_\Gamma \diff \mu=\int_M\omega.$
}
\end{definition}

\begin{theorem}\label{thm:sdiff-functions}{\rm \cite{IKM}}
The mapping  assigning the measured Reeb graph $\Gamma_F$ to a simple Morse function $F$ provides a one-to-one correspondence between simple Morse functions on $M$ up to a symplectomorphism 
and measured Reeb graphs compatible with $M$.
\end{theorem}

 
 \medskip
 \subsection{Antiderivatives on graphs}\label{sect:antider}

 In order to pass from the above classification of simple Morse functions on symplectic surfaces to the classification of coadjoint orbits of the group $\SDiff(M)$, we need to introduce the notion of the antiderivative of a density on a graph. Let $\Gamma$ be an oriented graph. Let also $ \rho$ be a \textit{density} on $\Gamma$, i.e. a finite signed Borel measure.
 
 \begin{definition}\label{circProperties}
A function $\lambda \colon \Gamma \,\setminus\, V \to \R$ defined and continuous on the graph $\Gamma$ outside its set of vertices $V = V(\Gamma)$ is called an \textit{antiderivative of the density $\rho$} if it has the following properties.
\begin{longenum}
\item It has at worst jump discontinuities at vertices, which means that for any vertex $v \in V$ and any edge $e \ni v$, there exists a finite limit
$
\lim\nolimits_{{x \xrightarrow[]{e} v }} \lambda(x),
$
where $ {{x \xrightarrow[]{e} v }} $ means ``as $x$ tends to $v$ along the edge $e$''.
\item Assume that $x,y$ are two interior points of some edge $e \in \Gamma$, and that $e$ is pointing from $x$ towards $y$. Then $\lambda$ satisfies the Newton-Leibniz formula
\begin{align}
\label{stokes}
\lambda(y) - \lambda(x) = \rho([x,y]).
\end{align}
\item For a vertex $v$ of $\Gamma$ the function $\lambda$ satisfies the Kirchhoff rule at $v$:
\begin{align}
\label{3valentcirc}
 \sum_{{e \to v}} \lim\nolimits_{{x \xrightarrow[]{e} v }} \lambda(x)= \sum_{{e \leftarrow v}} \lim\nolimits_{{x \xrightarrow[]{e} v }} \lambda(x)\,,
\end{align}
where the notation $e \to v$ stands for the set of edges pointing at the vertex $v$, and ${e \leftarrow v}$ stands for the set of edges pointing away from $v$. 
\end{longenum}
\end{definition}
 \begin{figure}[b]
\centerline{
\begin{tikzpicture}[thick]
    \node [vertex] at (7,-0.2) (nodeC) {};
    \node [vertex]  at (7,1.05) (nodeD) {};
    \node [vertex] at (7,2.75) (nodeE) {};
    \node [vertex]  at (7,4.2) (nodeF) {};
    \draw  [a] (nodeC) -- (nodeD);
    \fill (nodeC) circle [radius=1.5pt];
    \fill (nodeD) circle [radius=1.5pt];
    \fill (nodeE) circle [radius=1.5pt];
    \fill (nodeF) circle [radius=1.5pt];
  \draw [a] (nodeD) .. controls +(0.7, 0.5) and +(0.7,-0.5) .. (nodeE);
    \draw [a] (nodeD) .. controls +(-0.7, 0.5) and +(-0.7,-0.5) .. (nodeE);
    \draw  [a] (nodeE) -- (nodeF);
    \node at (7.8,1.9) () {{${e_3}$}};
      \node at (6.2,1.9) () {{${e_2}$}};
     \node at (7.3,3.5) () {{${e_4}$}};
          \node at (6.8,0.3) () {{${e_1}$}};
          \node at (7.2, -0.1) () {$0$};
                \node at (6.8, 4.1) () {$0$};
                    \node at (7.3, 0.85) () {$a_1$};
                            \node at (7.8, 1.2) () {$a_1 - z$};
                              \node at (6.6, 1.2) () {$z$};
                                   \node at (6.15, 2.6) () {$ a_2 + z$};
                                   \node at (8.25, 2.6) () {$a_1 + a_3 -z$};
                                    \node at (6.65, 3) () {$-a_4$};
\end{tikzpicture}
}
\caption{The space of antiderivatives on a graph of genus one.}\label{circTorus}
\end{figure}
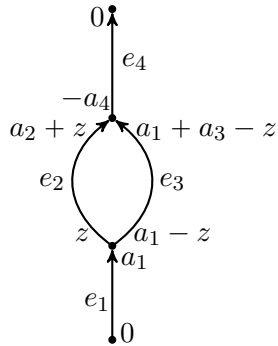
\begin{proposition}\label{circFuncs}
For an oriented graph  $\Gamma$  a density $\rho$ on $\Gamma$ admits an antiderivative if and only if
$\rho(\Gamma) = 0$.
Furthermore, if a density $\rho$ on $\Gamma$ admits an antiderivative, then the set of antiderivatives of $\rho$ is an affine space whose associated vector space is the  homology group $\Hom_1(\Gamma, \R)$.
\end{proposition}

\begin{example}\label{example:circtorus}

Consider the graph $\Gamma$ depicted in Figure \ref{circTorus}.  Let $\rho$ be a density on $\Gamma$ such that
$\rho(e_i) = a_i$, where the numbers $a_i$ satisfy $a_1 + a_2 + a_3 + a_4 = 0$ (so that the density $\rho$ admits an antiderivative). The numbers near vertices in the figure stand for the limits of the antiderivative $\lambda$ of $\rho$.  The space of such antiderivatives has one parameter
$z$ 
 (by the proposition above the space of antiderivatives is one-dimensional).
\end{example}

%

\medskip

\subsection{Classification of coadjoint orbits}\label{sect:coadj-sdiff}\label{sect:circ}
Let $M$ be a closed connected surface endowed with a symplectic form $\omega$.
Recall that the regular dual $\SVect^*(M)$ of the Lie algebra $\SVect(M)$ of divergence-free
vector fields on a surface $M$  is identified with the space $\Omega^1(M) / \diff \Omega^0(M)$ 
of smooth $1$-forms modulo exact $1$-forms on $M$. 
The coadjoint action of a $\SDiff(M)$ on $\SVect^*(M)$ is given by the change of coordinates in (cosets of) 1-forms on $M$ by means of a symplectic diffeomorphism:
$
\Ad^*_\Phi \,[\oneform] = [\Phi^*\oneform].
$
To describe orbits of the coadjoint action of $\SDiff(M)$ on $\SVect^*(M) $, consider the mapping
$
\mathfrak{curl} \colon \Omega^1(M) \,/\, \diff \Omega^0(M) \to C^\infty(M)
$
given by taking the vorticity function
$$
\mathfrak{curl}[\oneform] := \frac{\diff \oneform}{\omega}.
$$
(One can view this map as taking the vorticity  of a vector field $u=\alpha^\sharp$.)
Note  that the image of the mapping $\mathfrak{curl}$ is the space of functions with zero mean.

By definition, the mapping $\mathfrak{curl}$ is equivariant with respect to the $\SDiff(M)$ action: if cosets $[\oneform], [\oneformtwo] \in \SVect^*(M)$ belong to the same coadjoint orbit, then the functions $\mathfrak{curl}[\oneform]$ and $  \mathfrak{curl}[\oneformtwo]$ are related by a symplectic diffeomorphism. In particular, if  $\mathfrak{curl}[\oneform]$ is a simple Morse function, then  so is $  \mathfrak{curl}[\oneformtwo]$.

\begin{definition}
{\rm
We say that a coset of 1-forms $[\oneform] \in \SVect^*(M)$ is \textit{Morse-type} if $\mathfrak{curl}[\oneform]$ is a simple Morse function.
A coadjoint orbit $\orbit \subset \SVect^*(M)$ is \textit{Morse-type} if any coset $[\oneform] \in \orbit$ is Morse-type (equivalently, if at least one coset $[\oneform] \in \orbit$ is Morse-type).
}
\end{definition}
\par
Let $[\oneform] \in \SVect^*(M)$ be Morse-type, and let $F := \mathfrak{curl}[\oneform]$.   
Consider the measured Reeb graph $\Gamma_F$. Since $\mathfrak{curl}$ is an equivariant mapping, this graph is invariant under the coadjoint action of $\SDiff(M)$ on $\SVect^*(M)$. However, this invariant is not complete if $M$ is not simply connected (i.e., if $M$ is not a sphere $S^2$). To construct a complete invariant, we  endow the graph $\Gamma_F$ with a \textit{circulation function} constructed as follows. Let $\pi \colon M \to \Gamma_F$ 
be the natural projection. Take any point $x$ lying in the interior of some edge 
$e \in \Gamma_F$. Then $\pi^{-1}(x)$ is a circle. It is naturally oriented 
as the boundary of the set of smaller values of $F$. The integral of $\oneform$ 
over $\pi^{-1}(x)$ does not depend on the choice of a representative $\oneform \in [\oneform]$.
Thus, we obtain a function 
$
\circulation \colon \Gamma_F\, \setminus \, V( \Gamma_F) \to \R
$
given by 
\begin{align}\label{cosetCirculation}
\circulation(x) := \int\nolimits_{\pi^{-1}(x)} \!\!\oneform\,.
\end{align}
Note that in the presence of a metric on $M$, the value $\circulation(x)$ is the circulation over the level 
$\pi^{-1}(x)$ of the vector field  $\alpha^\sharp$ dual to the 1-form $ \oneform$.
\begin{proposition}\label{circIsAntiDer}
For any Morse-type coset $[\alpha] \in \SVect^*(M)$, the function $\circulation$ given by formula~\eqref{cosetCirculation} is an antiderivative of the density $\rho(I) := \int_I f \diff \mu$ in the sense of Definition \ref{circProperties}.
\end{proposition}
\begin{remark}
This density $\rho$ is the pushforward of the vorticity $2$-form $\diff[\alpha]$ from the surface to the Reeb graph.
\end{remark}
\begin{proof}[Proof of Proposition \ref{circIsAntiDer}]
The proof is straightforward and follows from the Stokes formula and additivity of the circulation integral.
\end{proof}
\begin{definition}
Let $(\Gamma, f, \mu)$ be a measured Reeb graph. A \textit{circulation function} $\circulation$ on $\Gamma$ is an antiderivative of the density $\rho(I) := \int_I f \diff \mu$.
A measured Reeb graph endowed with a circulation function is called a \textit{circulation graph}. 
\end{definition}
So, with any Morse-type coset $[\alpha] \in \SVect^*(M)$ we associate a circulation graph $\Gamma_{[\alpha]}$. 

\begin{theorem}{\rm \cite{IKM}}\label{thm:thm4-detailed}
Let $M$ be a compact connected symplectic  surface. Then Morse-type coadjoint orbits of $\SDiff(M)$ are in one-to-one correspondence with circulation graphs  
$(\Gamma, f, \mu, \circulation)$ compatible with $M$.
In other words, the following statements hold:
\begin{longenum}
\item For a symplectic surface $M$  Morse-type cosets $[\oneform], [\oneformtwo] \in \SVect^*(M)$ 
 lie in the same orbit of the $\SDiff(M)$ coadjoint action if and only if circulation graphs $\Gamma_{[\oneform]}$ and $\Gamma_{[\oneformtwo]}$ corresponding to these cosets are isomorphic.
\item For each circulation graph $\Gamma$ which is compatible\footnote{See Definition \ref{def:compatible} for compatibility of a graph and a surface.}
 with $M$, there exists a Morse-type $[\oneform] \in   \SVect^*(M)$ such that  $\Gamma_{[\oneform]} =(\Gamma, f, \mu, \circulation)$.
\end{longenum}
\end{theorem}

\begin{remark}
The space of circulation graphs for a given vorticity function has dimension equal to $\dim H_1(\Gamma, \R)={\rm genus}(M)$. Given a vorticity function, in order to define a circulation 
graph uniquely it suffices to consider a measured graph and set the values of $\varkappa$ circulations, 
where $\varkappa=\dim H_1(\Gamma, \R)$, one value on each cycle of $\Gamma$. 
In other words, one can consider $\varkappa$ points $p_i, \, i=1,...,\varkappa$ on the graph $\Gamma$ 
so that cuts at those points turn $\Gamma\,\setminus\, \{p_1, \dots, p_\varkappa\}$ into a tree, i.e. graph without cycles. 
Then by prescribing values of an antiderivative at all those points $p_i$ we 
determine the circulation function on the graph uniquely.
\end{remark}

\begin{remark}
For fluid dynamics we are interested only in connected components of coadjoint orbits, i.e. orbits with respect to 
the group $\SDiff_0(M)$, which is the connected component of the identity in the group $\SDiff(M)$ of 
all symplectomorphisms of $M$. To classify orbit invariants for the connected group $\SDiff_0(M)$ one needs to 
supplement invariants for $\SDiff(M)$ given by Theorem \ref{thm:thm4-detailed} 
by adding certain discrete invariants 
related to pants decompositions of the surface $M$ and Dehn half-twists. A complete list of the corresponding invariants
is given by Theorem 4.7 in \cite{IKM}, which we refer to for more detail. Note that those discrete invariants 
do not affect the list of Casimirs we are interested in here. 
\end{remark}

\medskip

\section{Casimir invariants of the 2D Euler equation}\label{sect:Casimirs}
Above we classified coadjoint orbits of the group $\SDiff(M)$ in terms of graphs with certain additional structures, 
see Theorem \ref{thm:thm4-detailed}. However, for applications, it is important to describe numerical invariants of the coadjoint action, i.e., Casimir functions. We begin with the description of such invariants for functions on symplectic surfaces.\par
Let $(M, \omega)$ be a closed connected symplectic surface, and let $F$ be a simple Morse function on $M$. 
With each edge $e$ of the measured Reeb graph $\Gamma_F = (\Gamma, f ,\mu)$, one can associate 
an infinite sequence of moments
$$
m_{i,e}(F) = \int_e f^i \,\diff \mu = \int_{M_e}\!\! F^i \,\omega\,,
$$
where $ i =0,1,2,\dots$, and $M_e = \pi^{-1}(e)$ for the natural projection $\pi \colon M \to \Gamma$. Obviously, the moments $m_{i,e}(F) $ are invariant under the action of $\SDiff(M)$ on simple Morse functions. Moreover, they form a complete set of invariants in the following sense:

\begin{theorem}\label{thm:moments}
Let $(M, \omega)$ be a closed connected symplectic surface, and let $F$ and $G$ be simple Morse functions on $M$. Assume that $\phi \colon \Gamma_F \to \Gamma_G$ is an isomorphism of abstract directed graphs which preserves moments on all edges. Then $\Gamma_F$ and $\Gamma_G$ are isomorphic as measured Reeb graphs, and there exists a symplectomorphism $\Phi \colon M \to M$ such that $\Phi_*F = G$.
\end{theorem}

\begin{proof}
Consider an edge $e = [v ,w] \in \Gamma_F$. Pushing forward the measure $\mu$ on $e$ by means of the homeomorphism $f \colon e \to [f(v), f(w)] \subset \R$, we obtain a measure $\mu_f$ on the interval $I_f = [f(v), f(w)]$, whose moments coincide with the moments of $\mu$ at $e$. Repeating the same construction for the measure on $\Gamma_G$, we obtain another measure $\mu_g$, which is defined on the interval $I_g = [g(\phi(v)), g(\phi(w))]$ and has the same moments as $\mu_f$. \par Now, consider any closed interval $I \subset \R$ which contains both $I_f$ and $I_g$. Then the measures $\mu_f$, $\mu_g$ may be viewed as measures on the interval $I$ supported at $I_f$ and $I_g$ respectively. The moments of the measures $\mu_f, \mu_g$ on $I$ coincide, so, by the uniqueness theorem for the Hausdorff moment problem (see Remark \ref{rem:hausdorff}), we have $\mu_f = \mu_g$, which implies the proposition.
\end{proof}

\begin{remark}\label{rem:hausdorff}
The Hausdorff moment problem gives the following necessary and sufficient condition: a sequence
of numbers $m_k$ can be the set of moments $m_k(\lambda)=\int_0^1\lambda^k \,\diff \mu(\lambda)$ of some 
Borel measure $\mu$ supported on the interval $[0,1]$ if and only if it satisfies the so-called monotonicity 
conditions. The latter are linear inequalities on $m_k$, which can be derived from the relations 
$\int_0^1\lambda^k (1-\lambda)^n\,\diff \mu(\lambda)\ge 0$ for all integer $k, n\ge 0$, where the left-hand side
is expressed in terms of $m_k$. (For instance, 
$m_3-2m_4+m_5=\int_0^1\lambda^3 (1-\lambda)^2\,\diff \mu(\lambda)\ge 0$.) In our case, replacing $\lambda$ 
by the parameter $f$ we only employ the statement that the measure $\mu(f)$ is fully determined by the set 
$\{m_k, k=0,1,2,\dots\}$. 

In fact, it turns out that under certain regularity conditions the measure $\mu$ 
can be found in a constructive way from the moment sequence $\{m_k\}$.
Assume, e.g. that $ \mu(\lambda)$ is supported on a segment $[-L,L]$ and is given by a smooth positive density 
function $ \diff \mu(\lambda)= w(\lambda)\,\diff \lambda$. 
Then consider the function $\Phi$ of a complex variable $\lambda$ defined by
$$
\Phi(\lambda)=\int_{[-L,L]}\frac{\diff \mu(z)}{\lambda - z}=\sum_{k\ge 0} \frac{m_k}{\lambda^{k+1}}\,.
$$
The integral expression shows that $\Phi$ is defined and holomorphic in the complement of the real segment 
$[-L,L]\subset \R$. (One can also show that $|m_k|\le CL^k$ and hence the series converges for $|\lambda|>L$.)
Now the measure density $w(\lambda)$ can be recovered from $\Phi$ as its normalized jump across the cut 
$[-L,L]$ in the real axis (see \cite{Akhiezer, Segal}):
$$
 w(\lambda)=\frac{1}{2\pi i}\lim_{\epsilon\to 0}(\Phi(\lambda-i\epsilon)-\Phi(\lambda+i\epsilon))\,.
$$
\end{remark}

\medskip

The above Theorem \ref{thm:moments} allows one to describe Casimirs of the 2D Euler equation on $M$, i.e. invariants of the coadjoint action 
of the symplectomorphism group $\SDiff(M)$. 
Let $F$ be a Morse vorticity function of an ideal flow with velocity $v$ on a closed surface $M$, and let
$\Gamma$ be its Reeb graph. The corresponding moments  $m_{i,e}(F) $ for this vorticity are natural to call 
{\it generalized enstrophies}. Then group coadjoint orbits in the vicinity of an
orbit with the vorticity function $F$ are singled out as follows.

\begin{corollary}\label{cor:Casimirs}
A  complete set of Casimirs of a 2D Euler equation in a neighborhood of a Morse-type coadjoint orbit is  given 
by the moments $m_{i,e}$ for each edge $e \in \Gamma$,  $i=0,1,2,\dots$, and all circulations of the velocity $v$
over cycles in the singular levels of $F$ on $M$.
\end{corollary}

Note that the (finite) set of required circulations can be sharpened by considering fewer quantities 
needed to describe the circulation function, as in Section \ref{sect:coadj-sdiff}.


\begin{remark}
As invariants of the coadjoint action of $\SDiff(M)$, one usually considers total moments
$$
m_i(F)= \int_{M}\!\! F^i \,\omega\ = \int_{\Gamma} f^i \,\diff \mu \, ,
$$
where $F = \Diff[\alpha]$ is the vorticity function, and $(\Gamma,f,\mu)$ is the measured Reeb graph of $F$. However, the latter moments do not form a complete set of invariants even in the case of a sphere or a disk.

 \begin{figure}[t]
 \centering
\begin{tikzpicture}[thick, scale = 1.3]
\node [vertex] (A) at (4.9,0.7) {};
\node[vertex] (B) at (5.5,-0.3) {};
\node[] () at (4.5,0.6) {$\Gamma$};
\node[vertex] (C) at (6,1) {};
\node[vertex] (D) at (5.5,-1) {};
    \fill (A) circle [radius=1.5pt];
    \fill (B) circle [radius=1.5pt];
        \fill (C) circle [radius=1.5pt];
            \fill (D) circle [radius=1.5pt];
    \draw [a] (B) -- (A);
        \draw [a] (B) -- (C);
            \draw [a] (D) -- (B);
            \draw [->] (8,-1.2) -- (8,1.5);
            \draw [dotted] (5.32, 0) -- (7.95,0);
                  \draw [dotted] (5.1, 0.4) -- (7.95,0.4);
                   \fill (8,0) circle [radius=1pt];
                   \draw [line width=1.5pt] (5.27,0) -- (5.37, 0);
                      \draw  [line width=1.5pt] (5.57,0) -- (5.67, 0);
                      
                      \draw  [line width=1.5pt] (5.03,0.4) -- (5.13, 0.4);
                          \draw  [line width=1.5pt] (5.73,0.4) -- (5.83, 0.4);
                     \fill (8,0.4) circle [radius=1pt];
                 \node () at (8.15, 1.3) {$f$};
                 \node () at (8.12, 0) {$a$};
                       \node () at (8.12, 0.4) {$b$};
                         \node () at (5, 0.2) {$I_1$};
                          \node () at (5.9, 0.2) {$I_2$};
\end{tikzpicture}
\caption{Modifying the measure on the edges.}\label{modMeasure}
\end{figure}
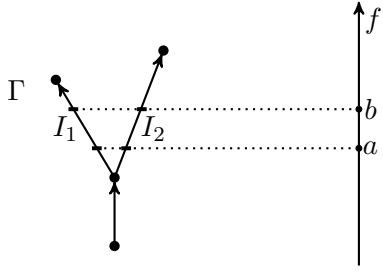
Consider, for example, the measured Reeb graph $(\Gamma, f ,\mu)$ depicted in Figure \ref{modMeasure}. Let $\mu'$ be any smooth measure on $\R$ supported in $[a,b]$. Define a new measure $\tilde \mu$ on $\Gamma$ by ``moving some density from one branch to another", i.e. by setting
\begin{align*}
\tilde \mu := \begin{cases} \mu + f^*(\mu') \mbox{ in } I_1,\\
 \mu - f^*(\mu') \mbox{ in } I_2,
\end{cases}
\end{align*}
and $\tilde \mu := \mu$ elsewhere. Then $(\Gamma, f ,\tilde \mu)$ is a again a measured Reeb graph. Moreover, for all total moments we have
$$
\int_\Gamma f^k\,\diff \mu = \int_\Gamma f^k\,\diff \tilde \mu\,.
$$
However, the measured graphs $(\Gamma, f ,\mu)$ and $(\Gamma, f ,\tilde \mu)$ are not isomorphic and thus correspond to two different coadjoint orbits of $\SDiff(S^2)$.
\end{remark}

\medskip 


\section{Examples and open questions}\label{sec:examples}

In this section we consider several examples and open questions related to the description of Casimirs for the case of a surface with boundary.

\subsection[Generalized enstrophies and circulations]{Generalized enstrophies and circulations}

In the next example we look at a genus one surface $M$ with an area form $\omega$ and 
a Morse height function $F$ on it as in Figure \ref{torusInt}.
Consider the domain $M_e\subset M$ associated with each edge $e$ of the Reeb graph $\Gamma_F$.
This domain $M_e=\pi^{-1}(e)$  is the preimage of the edge bounded by the corresponding critical levels of $F$. 
In this example there are 6 edges in the graph  $\Gamma_F$. 
Consider the infinite set of all generalized enstrophies, i.e. all 
moments  $\int_{M_e} \!\! F^k\,\omega$ of the vorticity function in these regions.
All generalized enstrophies are Casimirs of the coadjoint action of the group $\SDiff(M)$.

They do not exhaust all Casimirs, but in general must be supplemented  by several circulations. 
In this example of a torus one needs to fix the  value of one circulation (and, more generally, 
one  value for each handle of the surface). For instance, one can fix a value of the circulation 
function at the lower boundary of domain $M_e$, corresponding to the bottom of edge $e$ in Figure \ref{torusInt}.
(In Corollary \ref{cor:Casimirs} we mentioned circulations over all critical levels of $F$, which contains 
the one above, and several other circulations which are dependent on it.) Note that the set of all 
generalized enstrophies and circulations described in the corollary is not a ``minimal set" of Casimirs, as 
the Hausdorff moment problem does not claim the minimality.

\begin{remark}
Recall that the dual   $\mathfrak g^*={\SVect}^*(M)$ to the Lie algebra 
$\mathfrak g={\SVect}(M)$ of divergence-free vector fields on $M$ consists of cosets of 1-forms
$[\alpha]=\{\alpha+df\,|\,\text{ for all } f\in C^\infty(M)\}$, elements of the quotient 
$ \Omega^1(M) / \diff \Omega^0(M)={\SVect}^*(M)$.
The  function $F=\diff \wave \alpha/\omega$, as well as values 
$\circulation(x) := \int\nolimits_{\ell_x} \!\!\wave \alpha$, are defined for any 
$[\wave \alpha] \in \mathfrak g^*={\SVect}^*(M)$ in a neighborhood of the coadjoint orbit of $[\alpha]$.
Their invariant definition, 
i.e. the definition relying only on the choice of an area form, but not a metric on $M$, in a sense,
explains  the Casimir property of those quantities. On the other hand, the {\it interpretation} of those 
values as the set of {\it vector fields with given vorticity and circulations} (rather than the set of 1-forms), which are 
metric-related to the 1-forms, requires the presence of metric, and hence such a set is not  $\SDiff(M)$-invariant. 
\end{remark}



\subsection{The boundary case: Morse functions}\label{sec:bcmf}

Here we briefly describe the necessary changes in the classification theorems in the case of a surface $M$ 
with boundary $\partial M$ and main difficulties which arise in this setting. 
Now the group of area-preserving diffeomorphisms  of a connected surface $M$ has 
the Lie algebra consisting of divergence-free vector fields on $M$ tangent to  $\partial M$. 
As before, we first try to classify generic functions on $M$ with respect to this group action,
and then move to coadjoint orbirs.

\begin{definition}
Let $M$ be a compact connected surface with a possibly non-empty boundary $\partial M$.
A Morse function  $F \colon M \to \R $ is called \textit{simple} if it satisfies the following conditions:

\begin{longenum}
\item $F$ does not have critical points at the boundary;
\item the restriction of $F$ to the boundary $\partial M$ is a Morse function;
\item all critical values of $F$ and its restriction $F|_{\partial M}$ are distinct\footnote{Slightly more generally, in agreement with Definition \ref{def:smf}, one can assume that any $F$-level contains either at most one critical point or at most one critical point of the restriction (but not both).}.\end{longenum}
\end{definition}

Associate the following {\it Reeb graph} $\Gamma_F$ defined as the space of $F$-levels to such 
a simple Morse function $F \colon M \to \R$. 
Each vertex of this graph
corresponds either to a critical level of the function $F$ or to a critical point  of its restriction $F|_{\partial M}$ to 
the boundary. As before, the function $F$ on $M$ descends to a function $f$ on the graph $ \Gamma_F$,
which allows us to orient the edges of $\Gamma_F$ by increasing $f$. 

Note that now noncritical levels of $F$ are either circles or segments. 
We denote the corresponding edges of the Reeb graph  by {\it solid lines} if they correspond to {\it circle levels} 
and by {\it  dashed lines} if they correspond to {\it segment levels.} 
In the boundary case, in addition to two types of vertices for solid lines, max/min
 corresponding to 1-valent vertices, and saddles corresponding to $Y$-type 3-valent vertices, we have 
 5 more types of vertices involving dashed lines.
 
 \smallskip
 \begin{figure}[t]
{
\begin{tabular}{m{1cm}|m{3cm}|m{2cm}|m{1cm}|m{3cm}|m{2cm}}
\centering{Type} & \centering{Level Sets} & \centering{Reeb Graph} & \centering{Type} & \centering{Level Sets} & {\centering Reeb Graph}
\\
\hline
 \centering I & \centering
\begin{tikzpicture}[thick, scale = 1.33]
  \draw (-0.91, 0) -- (0.91, 0);
  \draw [dotted] (0.3, 0) .. controls (0.2, 0.3) and (-0.2, 0.3) .. (-0.3,0);
    \draw [dotted](0.6, 0) .. controls (0.4, 0.6) and (-0.4, 0.6) .. (-0.6,0);
      \draw [dotted] (0.9, 0) .. controls (0.6, 0.9) and (-0.6, 0.9) .. (-0.9,0);
        \fill [opacity = 0.1] (-0.9, 0) -- (0.9, 0)  .. controls (0.6, 0.9) and (-0.6, 0.9) .. (-0.9,0) ;
            \fill (0, 0) circle [radius=1.5pt];
      \end{tikzpicture}
&
\centering
\begin{tikzpicture}[thick, scale = 2]
      \draw [a, dashed](1.5, 0) -- (1.5, 0.5);
       \fill (1.5, 0) circle [radius=1pt];
\end{tikzpicture}
&
\centering
II
&
\centering
\begin{tikzpicture}[thick, scale = 2]
  \draw [] (-0.32, 0) -- (0.32, 0);
  \draw [dotted] (0, 0.4) circle [radius = 0.4];
          \fill (0, 0) circle [radius=1pt];
            \begin{scope}
    \clip (-0.6,0) rectangle (0.6, 1);
     \draw [dotted] (0, 0.4) circle [radius = 0.5];
       \fill [opacity = 0.1] (0, 0.4) circle [radius = 0.5];
              \fill [ fill = white] (0, 0.4) circle [radius = 0.3];
              \draw [dotted]  (0, 0.4) circle [radius = 0.3];;
\end{scope}
\end{tikzpicture}
&
{
{
\parbox{2cm}{
\centering
\begin{tikzpicture}[thick, scale = 2]
     \draw [a, dashed](1.5, 0) -- (1.5, 0.5);
         \draw [a](1.5, 0.5) -- (1.5, 1);
       \fill (1.5, 0.5) circle [radius=1pt];
\end{tikzpicture}
}
}
}
\\
\hline
\centering III & 
\centering
\begin{tikzpicture}[thick, scale = 1.33]
  \draw [] (-0.64, 0) -- (0.64, 0);
                     \fill (0, 0) circle [radius=1.5pt];
            \begin{scope}
    \clip (-1.1,0) rectangle (1.1, 0.8);
      \draw [dotted] (0, 0.8) circle [radius = 0.8];
     \draw [dotted] (0, 0.8) circle [radius = 1];
       \fill [opacity = 0.1] (0, 0.8) circle [radius = 1];
              \fill [fill = white, draw = white] (0, 0.8) circle [radius = 0.6];
                       \draw [dotted] (0, 0.8) circle [radius = 0.6];
\end{scope}
  \draw [] (-1.01, 0.8) -- (-0.6, 0.8);
  \draw [] (1.01, 0.8) -- (0.6, 0.8);   
\end{tikzpicture}
& 
\centering
\begin{tikzpicture}[thick, scale = 2]
              \draw [a, dashed](1.2, 0) -- (1.5, 0.5);
                  \draw [a, dashed](1.8, 0) -- (1.5, 0.5);
         \draw [a, dashed](1.5, 0.5) -- (1.5, 1);
       \fill (1.5, 0.5) circle [radius=1pt];
\end{tikzpicture}
& 
\centering
IV
&
\begin{tikzpicture}[thick, scale = 1.33]
                     \fill (0, 0) circle [radius=1.5pt];
                     \begin{pgfinterruptboundingbox}
                     \draw [dotted] (0,0) .. controls (-1.5,1.5) and (-1.5,-1.5) .. (0,0);
                     \end{pgfinterruptboundingbox}
                        \begin{scope}
                         \clip (-1.3, -1) rectangle (0.65, 1);
                         \begin{pgfinterruptboundingbox}
                                          \draw [dotted] (0,0) .. controls (1.5,1.5) and (1.5,-1.5) .. (0,0);
                                          \end{pgfinterruptboundingbox}
                     \draw [dotted] plot [smooth cycle] coordinates {(0, 0.2) (-0.65, 0.55) (-1.1, 0.4) (-1.25, 0) (-1.1, -0.4) (-0.65, -0.55) (0,-0.2) (0.65, -0.55) (1, -0.5) (1.2,0) (1, 0.5)(0.65, 0.55)};
                       \fill [opacity = 0.1] plot [smooth cycle] coordinates {(0, 0.2) (-0.65, 0.55) (-1.1, 0.4) (-1.25, 0) (-1.1, -0.4) (-0.65, -0.55) (0,-0.2) (0.65, -0.55) (1, -0.5) (1.2,0) (1, 0.5)(0.65, 0.55)};
                         \fill [fill = white] (0.65,0) ellipse (0.35 and 0.3);
                         \draw [dotted] (0.65,0) ellipse (0.35 and 0.3);
                     \end{scope}
                         \fill [fill = white] (-0.65,0) ellipse (0.35 and 0.3);
                         \draw [dotted] (-0.65,0) ellipse (0.35 and 0.3);;
                          \draw [] (0.65, 0.285) -- (0.65, 0.56);
                                   \draw [] (0.65, -0.285) -- (0.65, -0.56);
\end{tikzpicture}
&
\parbox{2cm}{
\centering
\begin{tikzpicture}[thick, scale = 2]
              \draw [a](1.2, 0) -- (1.5, 0.5);
                  \draw [a, dashed](1.8, 0) -- (1.5, 0.5);
         \draw [a, dashed](1.5, 0.5) -- (1.5, 1);
       \fill (1.5, 0.5) circle [radius=1pt];
\end{tikzpicture}
}
\\
\hline
\centering V
&
\centering
\begin{tikzpicture}[thick, scale = 1.33]
                     \fill (0, 0) circle [radius=1.5pt];
                        \begin{scope}                  
                         \clip (-0.65, -1) rectangle (0.65, 1);
                                              \begin{pgfinterruptboundingbox}
                     \draw [dotted] (0,0) .. controls (-1.5,1.5) and (-1.5,-1.5) .. (0,0);
                     \end{pgfinterruptboundingbox}
                         \begin{pgfinterruptboundingbox}
                                          \draw [dotted] (0,0) .. controls (1.5,1.5) and (1.5,-1.5) .. (0,0);
                                          \end{pgfinterruptboundingbox}
                     \draw [dotted] plot [smooth cycle] coordinates {(0, 0.2) (-0.65, 0.55) (-1.1, 0.4) (-1.25, 0) (-1.1, -0.4) (-0.65, -0.55) (0,-0.2) (0.65, -0.55) (1, -0.5) (1.2,0) (1, 0.5)(0.65, 0.55)};
                       \fill [opacity = 0.1] plot [smooth cycle] coordinates {(0, 0.2) (-0.65, 0.55) (-1.1, 0.4) (-1.25, 0) (-1.1, -0.4) (-0.65, -0.55) (0,-0.2) (0.65, -0.55) (1, -0.5) (1.2,0) (1, 0.5)(0.65, 0.55)};
                         \fill [fill = white] (0.65,0) ellipse (0.35 and 0.3);
                         \draw [dotted] (0.65,0) ellipse (0.35 and 0.3);;
                           \fill [fill = white] (-0.65,0) ellipse (0.35 and 0.3);
                            \draw [dotted] (-0.65,0) ellipse (0.35 and 0.3);;
                     \end{scope}                    
           \draw [] (0.65, 0.285) -- (0.65, 0.56);
                                   \draw [] (0.65, -0.285) -- (0.65, -0.56);
                               \draw [] (-0.65, 0.285) -- (-0.65, 0.56);
                                   \draw [] (-0.65, -0.285) -- (-0.65, -0.56);
\end{tikzpicture}
&
\parbox{2cm}{
\centering
\begin{tikzpicture}[thick, scale = 2]
              \draw [a, dashed](1.2, 0) -- (1.5, 0.5);
                  \draw [a, dashed](1.8, 0) -- (1.5, 0.5);
         \draw [a, dashed](1.5, 0.5) -- (1.2, 1);
             \draw [a, dashed](1.5, 0.5) -- (1.8, 1);
       \fill (1.5, 0.5) circle [radius=1pt];
\end{tikzpicture}
}
\end{tabular}
}
\caption{Five types of critical points on surfaces with boundary and type of vertices of the corresponding Reeb
 graphs. Solid lines correspond to pieces of the boundary, while dotted lines are connected components of level sets of the function.}\label{types}
\end{figure}

 Namely, as depicted on Figure \ref{types} by employing this correspondence of solid and dashed edges of a Reeb graph  to circular levels and segments respectively,  one can have 
\begin{longenum}
\item a min/max on the boundary, corresponding to a 1-valent vertex with a dashed line in the Reeb graph, 
 
\item  a min/max on the boundary, corresponding to a 2-valent vertex with one solid 
 and one dashed line in the Reeb graph, 
 
\item a min/max on the boundary, corresponding to a $Y$-type 3-valent vertex with three dashed lines 
 in the Reeb graph, 
 
\item a saddle point on $M$, corresponding to a $Y$-type 3-valent vertex with two dashed lines 
 and a solid line coming together  in the Reeb graph, 
 
\item a saddle point on $M$, see Figure \ref{types}, corresponding a 4-valent vertex of $X$-type between 
 dashed edges in $\Gamma_F$.
\end{longenum}

%
%
%
%
%

\begin{definition}\label{RGDef2}
A \textit{Reeb graph} $(\Gamma,  f)$ is an oriented connected graph  $\Gamma$ with dashed and solid edges and
a continuous function 
$f \colon \Gamma \to \R$  which satisfy the  properties following from the above description of vertices.
\end{definition}

As before, a simple Morse function 
$F \colon M \to \R$ on an orientable connected surface $M$ with boundary can be associated with 
a Reeb graph $\Gamma_F$  in the sense of Definition \ref{RGDef}. 

\begin{example}\label{ex:two_surf}
It turns out one cannot reconstruct topology of the surface with boundary from its Reeb graph alone.
For example, consider a dashed graph $\Gamma$ with $\dim \Hom_1(\Gamma) = 2$, see Figure~\ref{fig:dg}. 
This graph corresponds to a disk with two holes, with the corresponding function given by the vertical 
coordinate $y$. Cutting the disk along the dashed level sets of $y$ and then restoring the three gluings with 
opposite orientations, one obtains a torus with one hole. (Indeed, after the new gluings one obtains 
an oriented surface with the same Euler characteristic $-1$,  but  with only one boundary component, hence 
a torus with a hole.)
Since we cut the surface along level sets, the torus is naturally equipped with a simple Morse function whose graph coincides with the initial one.

\end{example}


\begin{figure}
\centering
\begin{tikzpicture}[thick, scale = 1]
    \node [vertex] at (5.5,0.3) (nodeC) {};
    \node [vertex]  at (5.5,1) (nodeD) {};
    \node [vertex] at (5.5,3) (nodeE) {};
    \node [vertex]  at (5.5,3.7) (nodeF) {};
    \node [vertex]  at (5.1,1.7) (nodeG) {};    
                \node [vertex]  at (5.9,2.3) (nodeI) {};    
    \draw  [a, dashed] (nodeC) -- (nodeD);
    \fill (nodeC) circle [radius=2pt];
    \fill (nodeD) circle [radius=2pt];
    \fill (nodeE) circle [radius=2pt];
    \fill (nodeF) circle [radius=2pt];
        \fill (nodeG) circle [radius=2pt];
                \fill (nodeI) circle [radius=2pt];
    \draw  [a, dashed] (nodeE) -- (nodeF);
    \draw[a, dashed] (nodeG)  --  (nodeI);
     \draw[a, dashed] (nodeD)  .. controls +(0.6,+1) ..  (nodeE);     
          \draw[a, dashed] (nodeD)  .. controls +(-0.6,+1) ..  (nodeE);     
 \fill [opacity = 0.1] (8.5+1.5, 2) ellipse (1.7 and 1.7);
  \draw (8.5+1.5, 2) ellipse (1.7 and 1.7);
      \draw [densely dashed] (6.8+1.5, 2) -- (9.2+1, 2);
       \draw [densely dashed] (6.9+1.5, 1.5) -- (9.2+1, 1.5);  
    \fill [fill = white, draw = black] (9.2+1.5, 1.65) ellipse (0.65 and 0.65);
      \fill [fill = white, draw = black] (7.8+1.5, 2.35) ellipse (0.65 and 0.65);
\end{tikzpicture}
\caption{Dashed Reeb graph with $\dim \Hom_1(\Gamma) = 2$ corresponding to both a disc with two holes and torus with one hole. Cutting the disk drawn here along the three dashed levels and then restoring the gluings with 
opposite orientations, one obtains a torus with one hole.}\label{fig:dg}
\end{figure}
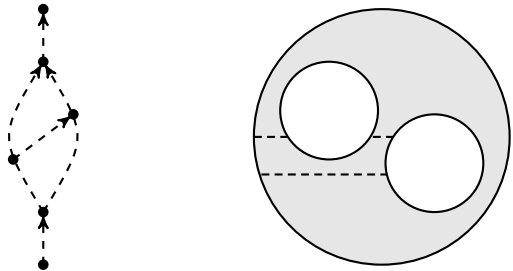

\begin{problem}
Describe the full information required from an abstract Reeb graph to reconstruct the corresponding 
surface with boundary.
\end{problem}

Further, we are interested in functions on surfaces equipped with an area form. 
As before,  the Reeb graph $\Gamma_F$ of any simple Morse function $F \colon M \to \R$ on a surface with 
 boundary has a natural structure of a measured Reeb graph, where the measure $\mu$ on $\Gamma_F$ is 
 the pushforward of the area form on $M$. In the presence of an area form  $\omega$ on $M$ 
 the Reeb graph gets endowed with an appropriately defined  \textit{log-smooth measure}.

\begin{problem}\label{measureproperty2} Describe the properties of the log-smooth
measure $\mu$ on a Reeb graph $(\Gamma,f)$ with solid and dashed edges. Namely,
describe the asymptotics similar to Definition \ref{measureproperty} at vertices of types ${\rm (i)}$-${\rm (v)}$.
 \end{problem}

Solution of this problem will allow one to define an abstract measured Reeb graph as a Reeb graph $(\Gamma,  f)$ 
with solid and dashed edges and endowed with a log-smooth measure $\mu$. The correspondence 
between abstract measured Reeb graphs and those corresponding to functions on surfaces with boundary 
now should include compatibility conditions for the surface and the Reeb graph beyond equality of total volumes 
 $\int_\Gamma \diff \mu=\int_M\omega$ and 
of the corresponding homology, as it should account for certain discrete information discussed above.

\begin{problem}
Describe the compatibility conditions for surfaces with boundary and  area form and augmented
measured Reeb graphs with solid and dashed edges.
 \end{problem}

Note that the would-be compatibility condition should be consistent with that for the case of Morse functions 
constant on boundary components and considered in \cite{IK} as a limiting case. (Formally speaking,
for simple Morse functions constants on the boundary have to be ruled out and can be considered only in the limit, since the restriction of functions to the boundary is to be Morse.)

The goal here is to establish a classification theorem similar to Theorem \ref{thm:sdiff-functions}:
For a surface $M$, possibly with boundary,
the mapping  assigning the augmented measured Reeb graph $\Gamma_F$ to a simple Morse function $F$ 
should provide a one-to-one correspondence between simple Morse functions on $M$ up to a symplectomorphism 
and augmented measured Reeb graphs compatible with $M$.

\medskip


\subsection{The boundary case:  coadjoint orbits}\label{sect:coadj-sdiff2}

To classify coadjoint orbits for the symplectomorphism group of 
a compact connected surface $M$  with boundary one can employ a classification of Morse functions on such surfaces. Recall that regardless of the boundary, the regular dual $\SVect^*(M)$ 
of the Lie algebra $\SVect(M)$ of divergence-free
vector fields on a surface $M$  is identified with the space $\Omega^1(M) / \diff \Omega^0(M)$. 
Similarly to the no-boundary case, for a  coset of 1-forms $\alpha\in \Omega^1(M) / \diff \Omega^0(M)$
consider the vorticity function $\mathfrak{curl}[\oneform] := {\diff \oneform}/{\omega}.$
Again confine ourselves to cosets of 1-forms $[\oneform] \in \SVect^*(M)$ of Morse type, i.e. to 
cosets with a simple Morse function  $\mathfrak{curl}[\oneform]$ on the surface with boundary.

\medskip

Let $[\oneform] \in \SVect^*(M)$ be Morse-type, and let $F = \mathfrak{curl}[\oneform]$. 
Now one can define a {\it circulation function} only on solid edges of the measured Reeb graph $\Gamma_F$. 
Indeed, let $\pi \colon M \to \Gamma_F$ 
be the natural projection. Take any point $x$ lying in the interior of a solid edge 
$e \in \Gamma_F$. Then $\pi^{-1}(x)$ is an oriented circle $\ell_x$.  The integral of $\oneform$ 
over $\ell_x$ does not depend on the choice of a representative $\oneform \in [\oneform]$.
(Note that preimages of points $x$ in dashed edges $e$ are segments, and integrals 
of $\oneform$  over them do depend on a representative $\oneform\in [\oneform]$, i.e. they are not well-defined objects on $\SVect^*(M)$.) 
Thus, we obtain a real-valued function $\circulation$ defined on solid edges, which is an antiderivative 
of the density $f \diff \mu$ wherever it is defined and which satisfies the properties of Definition \ref{circProperties} 
at the vertices involving only solid edges. 
It is suggestive to define a circulation graph as an augmented measured Reeb 
graph equipped with a circulation function. It turns out however that the circulation function  has to be 
further supplemented by additional information.

\begin{example}
Circulations, being integrals of 1-forms $\alpha$, over boundary components are invariants of coadjoint action, i.e.
Casimirs. (Since diffeomorphisms may interchange boundary components, the corresponding circulations 
over boundary components are to be considered up to permutation.) However, they do not enter the definition 
of circulation function on solid edges defined above.
 In particular, different surfaces corresponding to the same Reeb graph in Example \ref{ex:two_surf},
 see Figure \ref{fig:dg}, have different number of boundary components, and hence different number of 
 Casimirs from boundary conditions. On the other hand, since the Reeb graph consists 
 of only dashed edges, one does not relate a circulation function to it, and hence one needs new options 
 for ``storing" the information about the coadjoint orbits.
 \end{example}
 
 \begin{problem}
 How to supplement the circulation function on solid edges by extending it to 
 dashed ones to fully describe coadjoint orbits by means of ``circulation augmented measured Reeb 
graphs"?
\end{problem}

 One possibility would be to define circulation functions as  integrals of cosets over level segments (rather than 
 circle levels) by closing up the segments using boundary arcs.  
 We expect that the latter approach works at least in the case of one boundary component.

\begin{remark}
Note that a partial list of Casimirs in the 2D case with boundary can be described exactly 
in the same way as in the no-boundary case:
one has to consider all moments for the measure on each edge, either solid or dashed,  of the 
measured Reeb graph for the vorticity function. The discrete information on the Reeb graph does 
not affect the definition of families of continuous Casimirs. 
However, the list of additional circulations yet needs to be detailed.
\end{remark}

\bibliographystyle{plain}
\bibliography{sympl_orbits}
\addcontentsline{toc}{section}{References}
\end{document}